\documentclass[a4paper, 11pt]{amsart}
\usepackage{amsfonts, amsthm, amssymb, amsmath}
\usepackage{mathrsfs,array}
\usepackage{xy}
\usepackage{hyperref}
\usepackage{verbatim}
\usepackage[latin1]{inputenc}
\input xy
\xyoption{all}

\usepackage[titletoc]{appendix}

\newtheorem{theo}{Theorem}[section]
\newtheorem{prop}[theo]{Proposition}
\newtheorem{defi}[theo]{Definition}
\newtheorem{lemm}[theo]{Lemma}

\theoremstyle{definition}

\newcommand{\mb}{\mathbb}
\newcommand{\bb}{\mathbb}
\newcommand{\mc}{\mathcal}
\newcommand{\mf}{\mathfrak}

\newcommand{\ra}{\rightarrow}

\DeclareMathOperator{\cont}{cont}
\DeclareMathOperator{\Map}{Map}
\DeclareMathOperator{\GH}{GH}

\DeclareMathOperator{\Ext}{Ext}

\DeclareMathOperator{\Hom}{Hom}

\DeclareMathOperator{\Spec}{Spec}

\DeclareMathOperator{\Gal}{Gal}

\DeclareMathOperator{\Ind}{Ind}

\DeclareMathOperator{\SG}{SG}
\DeclareMathOperator{\GL}{GL}
\DeclareMathOperator{\SL}{SL}

\DeclareMathOperator{\LT}{LT}

\title{$p$-adic Jacquet-Langlands correspondence and patching}
\date {\today}
\author{Przemys\l aw Chojecki and Erick Knight}

\address{Przemys\l aw Chojecki\\Instytut Matematyczny PAN \\
ul. Sniadeckich 8 \\
00-656 Warszawa \\
Poland}
\email{pchojecki@impan.pl}
\urladdr{http://pchojecki.impan.pl/}

\address{Erick Knight \\ Harvard University \\
Science Center \\ 
USA}
\email{eknight@math.harvard.edu}
\urladdr{http://www.math.harvard.edu/~eknight/}

\begin{document}

\begin{abstract}
We describe two candidates for a local $p$-adic Jacquet-Langlands correspondence and using patching we show that they are in fact isomorphic. We then study locally algebraic vectors of the given correspondence.
\end{abstract}

\maketitle

\addtocontents{toc}{\protect\setcounter{tocdepth}{1}}
\tableofcontents

\section{Introduction}
\noindent Let $F$ be a finite extension of $\mb{Q}_p$. The goal of the local $p$-adic Langlands program is to establish a connection between $n$-dimensional $p$-adic Galois representations of the group $\textrm{Gal}(\bar{F} / F)$ and admissible Banach representations with an action of $\GL_n(F)$. After the initial success of Breuil, Berger, Colmez and others on establishing such a correspondence with desired properties in the case of $n=2$ and $F=\mb{Q}_p$, the progress was hold by different obstacles. The proof of Colmez, purely algebraic in nature, could not be generalized in a straighforward way, due to abundance of automorphic representations for $F \not = \mb{Q}_p$ (\cite{bp}). The proof of Harris and Taylor of the classical local Langlands correspondence used crucially geometrical input.

\medskip

\noindent New geometric methods, suitable for $p$-adic aspects of the Langlands program, became available recently with the rise of perfectoid spaces (\cite{sch3}). In \cite{sch} Scholze gave a construction of certain admissible representations of $D^{\times}$ (a division algebra) attached to admissible $\GL_n$-representations. He used local geometric methods and exploited the perfectoid structure of the Lubin-Tate space at infinity. Considering his work in two-dimensional case, we are able to extract an admissible 2-dimensional $D^{\times}$-representation $J'(\rho)$ attached to a local Galois representation $\rho: \Gal(\bar{\mb{Q}}_p / \mb{Q}_p) \ra \GL_2(E)$, where $E$ is a finite extension of $\mb{Q}_p$. 

\medskip

\noindent On the other hand, one of us (E.K.) in \cite{EKThesis}, gave a different construction for a possible $p$-adic Jacquet-Langlands correspondence, by exploiting Drinfeld tower and Cherednik uniformization of Shimura curves. This allowed him to attach an admissible $D^{\times}$-representation $J(\rho)$ to $\rho$. Our main theorem is an isomorphism of the two constructions. 

\begin{theo} Let $\rho :\Gal(\bar{\mb{Q}}_p / \mb{Q}_p) \ra \GL_2(E)$ be a continuous representation. We assume that the reduction $\bar{\rho}$ is not the sum of two characters, nor an extension of a character by itself, nor an extension of $\chi$ by $\chi \bar{\varepsilon}$ (where $\bar{\varepsilon}$ is the mod $p$ cyclotomic character and $\chi$ is any continous mod $p$ character). We have an isomorphism of $D^{\times}$-representations
$$J'(\rho) \simeq J(\rho)$$
\end{theo}

\noindent This shows that $J(-)$ is the natural candidate for the $p$-adic Jacquet-Langlands correspondence: it is functorial, compatible with local and global constructions and satisfies local-global compatibility.

\medskip

\noindent We prove it using patching (\cite{ceggps1}) and thus our proof is global in nature even though we start with local objects. By using patching and local-global compatibility results we are able to reduce the proof to easy cases when the isomorphism is clear.

\medskip

\noindent Another main input of this paper is the analysis of the locally algebraic vectors of $J(\rho)$ in certain cases. If $\rho : G_{\bb{Q}_p} \rightarrow \mathrm{GL}_2(E)$ is a continuous representation, then define $BS_{D^\times}(\rho)$ as follows: if $\rho$ is not potentially semistable with distinct Hodge-Tate weights, then $BS_{D^\times}(\rho) = 0$.  Otherwise, associated to $\rho$ are Hodge-Tate weights $w_1 < w_2$ and a Weil-Deligne representation $WD(\rho)$.  If the Frobenius-semisimplification of $WD(\rho)$ is the sum of two characters, then, again, define $BS_{D^\times}(\rho) = 0$.  In the final case, we let $Sm_{\rho}$ be the representation of $D^\times$ associated to $WD(\rho)^{F-ss}$ and $Alg_{\rho}$ be the algebraic representation of $D^\times$ with weights $-w_2$ and $-w_1 - 1$.  Then we may define $BS_{D^\times}(\rho) = Sm_{\rho} \otimes Alg_{\rho}$.

\begin{theo}
We have $J(\rho)^{alg} = BS_{D^\times}(\rho)$.
\end{theo}

\noindent We again prove this theorem using patching, which allows us to reduce the statement to globally arising representations.

\subsection{Notations}

We let $F/\mb{Q}_p$ be a finite extension with ring of integers $\mc{O}$ and a uniformiser $\varpi \in \mc{O}$. We identify the residue field with $\mb{F}_q$. Fix the algebraic closure $\bar{\mb{F}}_q$ and define $\breve{F} = F \otimes _{W(\mb{F}_q)} W(\bar{\mb{F}}_q)$ be the completion of the unramified extension of $F$ with residue field $\bar{\mb{F}}_p$. Let $\breve{\mc{O}} \subset \breve{F}$ be its ring of integers.

\medskip

\noindent Let $E/\mathbb{Q}_p$ be a finite extension with ring of integers $\mathcal{O}_E$, uniformizer $\varpi_E$, and residue field $k_E$. This will be our coefficient field.

\medskip

\noindent We will denote by $G_{\mb{Q}_p}$ the absolute Galois group $\Gal (\bar{\mb{Q}}_p / \mb{Q}_p)$ and similarly $G_{F_v} = \Gal (\bar{F}_v / F_v)$ for any finite extension $F_v / \mb{Q}_p$.

\medskip

\noindent We let $\Delta/\mathbb{Q}$ be a quaternion algebra that is split at $\infty$ and nonsplit at $p$. Put $G = \Delta^{\times}$; this is an algebraic group over $\mathbb{Q}$. We also let $D/\bb{Q}_p = \Delta(\bb{Q}_p)$ be the division algebra over $\bb{Q}_p$.

\medskip

\noindent We need to consider general rings as coefficient rings for our admissible and smooth representations, in order to apply results to patched modules. Following Emerton we have

\begin{defi}
Let $(A,\mf m)$ be a complete noetherian local ring with finite residue field of characteristic $p$ and $G$ be a $p$-adic analytic group. An $A[G]$-module $V$ is called smooth if for all $v \in V$ there is some open subgroup $H \subset G$ and $i \geq 1$ such that $v$ is $H$-invariant and $\mf m ^i v = 0$.

\medskip

\noindent A smooth $A[G]$-module $V$ is admissible if for all $i \geq 1$ and $H \subset G$ open, the $A/ \mf m ^i$ module $V^H[\mf m ^i]$ is finitely generated (equivalently, of finite length). 
\end{defi} 

\noindent By Proposition 2.2.13 of \cite{em2} the category of admissible $A[G]$-modules is abelian and it is a Serre subcategory of the category of smooth $A[G]$-modules.

\medskip

\noindent We recall that the $p$-adic local Langlands correspondence (in the formulation of Paskunas, cf. \cite{Pas}) associates to any admissible $A[\GL_2(\mb{Q}_p)]$-module $V$, a continuous $A[G_{\mb{Q}_p}]$-module $\rho _V$ of rank 2. Vice-versa, to any continuous $A[G_{\mb{Q}_p}]$-module $\rho$ of rank 2, we can associate an admissible $A[\GL_2(\mb{Q}_p)]$-module $B(V)$.



\section{Two constructions}

\subsection{First construction} Let us summarize the main results of \cite{EKThesis}.

\medskip

Let $\Omega_{\bb{Q}_p}/ \breve{\mb{Q}}_p$ be Drinfel'd's upper half plane.  There are covers $\Sigma^n/\Omega_{\bb{Q}_p}$ which are $\mathcal{O}_D^\times / (1 + \varpi_D^n \mathcal{O}_D)$-torsors.  The group $\mathrm{GL}_2(\bb{Q}_p)$ acts on $\Omega_{\bb{Q}_p}$, and it is possible to twist the action by $\phi^{\nu_p(\mathrm{det}(g))}$ so that the action of $\mathrm{GL}_2(\bb{Q}_p)$ lifts to the covers $\Sigma^n$.  There are natural commuting actions of $\mathcal{O}_D^\times$ and $I_{\bb{Q}_p}$ on the inverse limit as well, and these actions commute with the action of $\mathrm{GL}_2(\bb{Q}_p)$.  Finally, one may extend these actions by using the right power of Frobenius again so that one gets an action of $D^\times \times \mathrm{GL}_2(\bb{Q}_p) \times G_{\bb{Q}_p}$ on the tower $\Sigma^n$.

\begin{defi}
Let $(A,\mf m)$ be any complete noetherian local ring with finite residue field of characteristic $p$ and the fraction field $K$. Let $\hat{H}^i_{A}(\Sigma)$ be the $\mf m$-adic completion of $$\lim_{\longleftarrow \atop s} \lim_{\longrightarrow \atop n} H^i_{\acute{e}t}((\mathrm{Res}^{\hat{\bb{Q}}_p}_{\bb{Q}_p} \Sigma^n) \times_{\bb{Q}_p} \bb{C}_p, A/\mf m ^s A).$$  Additionally, let $\hat{H}^1_K(\Sigma) = \hat{H}^1_{A}(\Sigma) \otimes_{A} K$.
\end{defi}

All of the discussion about the tower $\Sigma^n$ is encoded in terms of properties of the space $\hat{H}^1_A(\Sigma)$ by saying that there are commuting unitary actions of $\mathrm{GL}_2(\bb{Q}_p)$, $D^\times$, and $G_{\bb{Q}_p}$ on $\hat{H}^1_A(\Sigma)$.

\medskip

Let $V$ be an admissible $A[\GL_2(\mb{Q}_p)]$-module. We can associate to it via the $p$-adic local Langlands correspondence a continuous $A[G_{\mb{Q}_p}]$-module $\rho _V$ of rank 2. We define
$$J(V) = \Hom _{G _{\mb{Q}_p}}\left(\rho _V, (\hat{H}^1_A(\Sigma) \widehat{\otimes} _{A} V)^{\GL_2(\mb{Q}_p)}\right)$$
This is a functor from the category of admissible smooth $A[\GL_2(\mb{Q}_p)]$-modules to the category of $A[D^{\times}]$-modules. Alternatively we can view it as a functor on the category of continuous $A[G_{\mb{Q}_p}]$-modules via
$$J(\rho) = \Hom _{G _{\mb{Q}_p}}\left(\rho, (\hat{H}^1_A(\Sigma) \widehat{\otimes} _{A} B(\rho))^{\GL_2(\mb{Q}_p)}\right)$$

\subsection{Second construction.} We now review the construction in \cite{sch}.

\medskip

One has the Lubin-Tate tower $(\mc{M}_{\LT,K}) _{K \subset \GL_n(F)}$ which is a tower of smooth rigid-analytic varieties $\mc{M} _{\LT,K}$ over $\breve{F}$ parametrized by compact open subgroups $K$ of $\GL_n(F)$ with finite \'etale transition maps. There is a compatible continuous action of $D^{\times}$ on all $\mc{M} _{\LT, K}$ and an action of $\GL_n(F)$ on the tower: each $g\in \GL_n(F)$ induces an isomorphism between $\mc{M} _{\LT,K}$ and $\mc{M} _{\LT, g^{-1}Kg}$.

\medskip

There is a map called the Gross-Hopkins period map
$$ \pi _{\GH} : \mc{M} _{\LT, K} \ra \mb{P} ^{n-1} _{\breve{F}}$$
compatible for varying $K$. It is an \'etale covering of rigid-analytic varieties with fibres being $\GL_n(F)/K$. Moreover it is $D^{\times}$-equivariant and there is a Weil descent datum on $\mc{M}_{\LT,K}$ under which $\pi _{\GH}$ is equivariant.

\medskip

Let us denote by $\mc{M}_{\LT, \infty}$ the perfectoid space over $\breve{F}$ constructed by Scholze (cf. \cite{sch}), which is the inverse limit in the adic setting, i.e.
$$\mc{M} _{\LT, \infty} \sim \varprojlim _K \mc{M}_{\LT, K}$$
We also have the Gross-Hopkins period map $\pi _{\GH} : \mc{M} _{\LT, \infty} \ra \mb{P} ^{n-1} _{\breve{F}}$, which can be viewed as a $\GL_n(F)$-torsor. We use this map to define sheaves associated to admissible representations.

\medskip

Let $\pi$ be an admissible $\mb{F}_p$-representation $\pi$ of $\GL_n(F)$. To each $D^{\times}$-equivariant \'etale map $U \ra \mb{P}^{n-1} _{\breve{F}}$ we can associate the $\mb{F}_p$-vector space
$$\Map _{\cont, \GL_n(F) \times D^{\times}} (|U \times _{\mb{P}^{n-1} _{\breve{F}}} \mc{M}_{\LT,\infty}|, \pi)$$
of continuous $\GL_n(F) \times D^{\times}$-equivariant maps.

\begin{prop}[\cite{sch}, Proposition 3.1] This association defines a Weil-equivariant sheaf $\mc{F}_{\pi}$ on $(\mb{P}^{n-1} _{\breve{F}} / D^{\times})_{et}$. The association $\pi \mapsto \mc{F} _{\pi}$ is exact and all geometric fibres of $\mc{F}_{\pi}$ are isomorphic to $\pi$.  
\end{prop}
   
The cohomology groups of $\mc{F}_{\pi}$ provide a good source of admissible representations. Let $C$ be an algebraically closed and complete extension of $F$.

\begin{prop}[\cite{sch}, Corollary 3.14] For any admissible smooth representation $\pi$ of $\GL_n(F)$ the cohomology group $H^i_{et}(\mb{P}^{n-1}_C, \mb{F}_{\pi})$ is an admissible $D^{\times}$-representation invariant under change of $C$.
\end{prop}

If $(A,\mf m)$ is any complete noetherian local ring with finite residue field of characteristic $p$ and $V$ is an admissible $A[\GL_n(F)]$-module, then we can attach to it $\mc{F}_V$ as before, getting a sheaf on $(\mb{P}^{n-1}_C / D^{\times})_{et}$. 

\begin{prop}[\cite{sch}, Theorem 4.4]
For all $i\geq 0$, the $D^{\times}$-representation $H^i(\mb{P}^{n-1}_C, \mc{F} _V)$ is admissible, independent of $C$ and vanishes for $i > 2(n-1)$.
\end{prop}

We now specify to the case $n=2$ and $F = \mb{Q}_p$, where we have a $p$-adic local Langlands correspondence $\rho \mapsto B(\rho)$. If $(A,\mf m)$ is any complete noetherian local ring with finite residue field of characteristic $p$ and $V$ is an admissible $A[\GL_2(\mb{Q}_p)]$-module, we can associate to it a continuous $A[G_{\mb{Q}_p}]$-module $\rho _V$ of rank 2. We define
$$J'(V) = \Hom _{G _{\mb{Q}_p}}(\rho _V, H^1(\mb{P}^1 _C, \mc{F}_V))$$
This is a functor from the category of admissible smooth $A[\GL_2(\mb{Q}_p)]$-modules to the category of $A[D^{\times}]$-modules. Alternatively we can view it as a functor on the category of continuous $A[G_{\mb{Q}_p}]$-modules via
$$J'(\rho) = \Hom _{G _{\mb{Q}_p}}(\rho , H^1(\mb{P}^1 _C, \mc{F}_{B(\rho)}))$$

\section{Local-global compatibilities}

Let $F/\bb{Q}$ be a totally real field where $p$ splits completely.  Choose one place $v$ over $p$ and let the others be $v_1, \ldots, v_n$.  Similarly, choose one place $w$ over infinity and let the others be $w_1, \ldots, w_n$.  It is convienent to introduce the notation $F_p^v = F_{v_1} \times \cdots \times F_{v_n}$.  Now choose a CM field $F'/F$ where $v$ and $v_i$ split for all $i$ (explicitly choose $v'$ and $v_i'$ such that $v = v' \overline{v}'$ and $v_i = v_i' \overline{v}_i'$), and is unramified at all finite places.  Let $\Delta/F'$ be a quaternion algebra that is split at all finite places away from $v$, and also has an involution $i$ of the second type such that, if $G = \{d \in \Delta | d i(d) = 1\}$ is the associated unitary group, then $G(F_v) = D^\times$, $G(F_{v_i}) = \mathrm{GL}_2(\mathbb{Q}_p)$, $G(F_w) = U(1,1)$ and $G(F_{w_i}) = U(2)$.  These assumptions imply that $F/\bb{Q}$ is an even degree extension.  $G$ will be viewed both as a group over $F$ as well as over $\bb{Q}$.  It is also useful to define $GG = \{d \in \Delta^\times | d i(d) \in \bb{Q}^\times\}$.  This is an algebraic group over $\bb{Q}$.  With this setup, it is possible to choose a division algebra $\overline{\Delta}/F'$ with an involution $\overline{i}$ of the second kind such that the associated unitary group $\overline{G}$ has the same invariants as $G$ away from $v$ and $w$, $\overline{G}(F_v) = \mathrm{GL}_2(\mathbb{Q}_p)$, and $\overline{G}(F_w) = U(2)$.  Finally, let $\overline{GG}$ be defined in analogy to $GG$.

\medskip

Let $K_p \subset GG(\bb{Q}_p)$ and $K^p \subset GG(\mathbb{A}_{f}^{p})$ be compact open subgroups.  Additionally, let $K^p_0$ be a hyperspecial subgroup of $GG(\mathbb{A}_F^{p\infty})$.  There is an isomorphism $GG(\bb{Q}_p) \cong \bb{Q}_p^\times \times G(F_v) \times G(F_{v_1}) \times \cdots \times G(F_{v_n})$.  We will abuse notation, and write $K_p = K_v K^v_p$ for $K_v \subset G(F_v)$ and $K^v_p \subset \bb{Q}_p^\times \times G(F_{v_1}) \times \cdots \times G(F_{v_n}) = \bb{Q}_p^\times \times D^\times \times \mathrm{GL}_2(\bb{Q}_p) \times \cdots \times \mathrm{GL}_2(\bb{Q}_p)$.  Since one has that $G(\mathbb{A}_F^{v\infty}) = \overline{G}(\mathbb{A}_F^{v\infty})$, there are compact open subgroups $\overline{K}_p^v$ and $\overline{K}^v$ corresponding to the subgroups $K_p^v$ and $K^p$.  Now, $GG$ gives rise to a Shimura curve, which will be denoted $Sh_{K_vK_p^vK^p} / F'$.  Also, there is a $\mathrm{GL}_2(\mathbb{Q}_p)$-set $X_{\overline{K}_p^v\overline{K}^p} := \overline{G}(F) \backslash \overline{G}(\mathbb{A}_F^\infty) / \overline{K}_p^v\overline{K}^p$.  As before, the Cerednik-Drinfel'd uniformization applies.  If $K_v = 1+\varpi_D^n \mathcal{O}_D$, then the uniformization can be written $Sh_{K_vK_p^vK^p, \mathbb{C}_p}^{an} = (\Sigma^n \times X_{\overline{K}_p^v\overline{K}^p}) / \mathrm{GL}_2(\mathbb{Q}_p)$, where $\mathrm{GL}_2(\bb{Q}_p)$ acts through its natural action on both $\Sigma^n$ and $X_{\overline{K}^v_p\overline{K}^p}$.

\medskip

To fix notation, let $\hat{H}^1_{\mathcal{O}_E, GG}(K^p) = \displaystyle{\lim_{\longleftarrow \atop s} \lim_{\longrightarrow\atop K_p}} H^1_{\acute{e}t}(Sh_{K_pK^p, \overline{F'}}, \mathcal{O}_E/\varpi_E^s)$.  Additionally, let $\hat{H}^1_{E, GG}(K^p) = \hat{H}^1_{\mathcal{O}_E, GG}(K^p) \otimes_{\mathcal{O}_E} E$.  There are also spaces $\hat{H}^0_{\mathcal{O}_E, \overline{GG}}(\overline{K}^p)$ and $\hat{H}^0_{E, \overline{GG}}(\overline{K}^p)$ defined similarly. We will write $\pi(\overline{K}^p)$ for $\hat{H}^0_{E, \overline{GG}}(\overline{K}^p)$ to make the notations lighter.

\medskip

We can now formulate the local-global compatibility for the first construction. For a Galois representation $\rho: G_{F} \ra \GL_2(E)$, let $J_{GG, \ell}(\rho |_{G_{F_\ell}})$ be the (generic) local Langlands correspondence if $GG(\bb{Q}_\ell) = \mathrm{GL}_2(\bb{Q}_\ell)$ and the Jacquet-Langlands correspondence otherwise. We define
$$AF_{\ell \neq p} =\left(\underset{\ell \neq p}{\bigotimes}' J_{GG, \ell}(\rho|_{G_{F_\ell}})\right)^{K^p}$$

\begin{theo}\label{theo-lg1}

Let $\rho: G_F \ra \GL_2(E)$ be a Galois representation pro-modular for $\overline{GG}$, with the associated Hecke character $
	\lambda: \mb{T}(K^p) \ra E$. We assume that for each $v|p$, the reduction $\bar{\rho}_v$ is not the sum of two characters, nor an extension of a character by itself, nor an extension of $\chi$ by $\chi \bar{\varepsilon}$. Then one has that
 $$\mathrm{Hom}_{G_{F}}(\rho, \hat{H}^1_{GG,E}(K^p)) \cong J(B(\rho|_{G_{F_{v}}})) \otimes_{i} B(\rho|_{G_{F_{v_i}}})  \otimes AF_{\ell \neq p}$$
\end{theo}
\begin{proof}
See the remarks following the proof of Theorem 5.2.2 in \cite{EKThesis}.
\end{proof}

\noindent Let us remind the reader the local-global compatibility for the group $\overline{GG}$, which was done in \cite{cs1} and \cite{cs2} using results of Emerton:

\begin{prop}\label{prop-lcbar}
	Let $\rho: G_F \ra \GL_2(E)$ be a Galois representation pro-modular for $\overline{GG}$, with the associated Hecke character $
	\lambda: \mb{T}(K^p) \ra E$. We assume that for each $v|p$, the reduction $\bar{\rho}_v$ is not the sum of two characters, nor an extension of a character by itself, nor an extension of $\chi$ by $\chi \bar{\varepsilon}$. Then we have as isomorphism of $\overline{GG}(\mb{A}_F)$-modules:
	$$\pi(K^p)^{\mb{T}(K^p)=\lambda} = \bigotimes _i B(\rho_{| G_{F_{v_i}}}) \otimes \Hom(\otimes _i B(\rho_{| G_{F_{v_i}}}), \pi(K^p))$$  	 
\end{prop}
\begin{proof}
The general method is due to Emerton and described in \cite{em1}, and then adapted in \cite{cs1} and \cite{cs2}. There are two steps: 

\medskip 

\noindent (1) Construct a non-zero map 
$$ \bigotimes _i B(\rho_{| G_{F_{v_i}}}) \ra \pi(K^p)^{\mb{T}(K^p)=\lambda}$$

\noindent (2) Show that such a map is an injection and in fact an isomorphism.

\medskip

Let us start with (1). This part is proved in \cite{cs1}, under the assumption that $\bar{\rho}_{v_i}$ is irreducible for each $i$. In a reducible totally indecomposable case, Lemma 5 of \cite{cs1} and the argument which follows, allow us to reduce the proof to the case of $\rho$'s lying in any Zariski dense set $S$. We take $S$ to be $P^{cris}_{autom}$ from Proposition 4.10 of \cite{cho3}, and thus $\rho_{| G_{F_{v_i}}}$ is crystalline and totally indecomposable at each $i$. Then the result follows from \cite{bh} or \cite{bc} (where the result is given for $U(3)$, but the same proof applies for $U(2)$ and is easier). 

\medskip

\noindent Let us now prove (2). By Proposition 1 and 2 of \cite{cs2} it is enough to show injectivity mod $p$ (in fact mod $\mf{m}$, where $\mf{m}$ is the maximal ideal of the Hecke algebra associated to the $\bar{\rho}$). As in \cite{cs2} we reduce the proof to showing that the any map
$$\bigotimes _i \pi _{v_i} \ra \pi(K^p) [\mf m]$$
is injective, where $\pi _{v_i}$ is the $\GL_2(F_{v_i})$-representation associated to $\rho_{| G_{F_{v_i}}}$ by the mod $p$ local Langlands correspondence. 

\medskip

The case which is not treated in \cite{cs1} and \cite{cs2}, is the case when some of $\bar{\rho} _{| G_{F_{v_i}}}$ are extensions of $\chi _i \bar{\varepsilon}$ by $\chi _i$. In this case $\pi _{v_i}$ is a non-split extension of the topologically irreducible $\Ind \chi _{v_i} \otimes \chi _{v_i} \bar{\varepsilon}$ by a representation which is itself a non-split extension of a one-dimensional representation by the topologically irreducible representation $(\chi _{v_i} \circ \det) \otimes \bar{St}$. This is a three-step filtration. Any map 
$$\bigotimes _i \pi _{v_i} \ra \pi(K^p) [\mf m]$$
cannot factor through the principal series quotient because the weights are wrong by Serre's conjecture (see \cite{bgg}), so the only possibility is that some of the maps factor through the Steinberg. But then one would have a 1-dimensional $\GL_2(F_{v_i})$-stable subspace in a non-Eisenstein part of completed cohomology. This contradicts Ihara's lemma as we now show.

\medskip

Let $\SG$ be the corresponding special unitary group over $F^+$. It suffices to show that a vector in $H^0(\SG(F^+)\backslash SG(A_{F^+, f})/K^p, k_E)$ which is $\SL_2(F^+_v)$-invariant is actually invariant under the action of $\SG(\mb{A}_{F^+, f})$.  But this is just strong approximation: we have invariance under the $\SG(F^+)$, $K^p$ and the points at a split place so we have invariance under the whole group. Going back to the usual unitary group, one gets that any one-dimensional subspace that is $G(F^+_v)$-stable must be $G(\mb{A}_{F^+, f})$-stable, and then any character of $G(\mb{A}_{F^+, f})$ must be Eisenstein. Thus the map must be injective.

\end{proof}

\noindent We can now prove the local-global compatibility for the group $GG$ for the second construction.

\begin{theo}\label{theo-lg2}
Let $\rho: G_F \ra \GL_2(E)$ be a Galois representation pro-modular for $GG$, with the associated Hecke character $
	\lambda: \mb{T}(K^p) \ra E$. We assume that for each $v|p$, the reduction $\bar{\rho}_v$ is not the sum of two characters, nor an extension of a character by itself, nor an extension of $\chi$ by $\chi \bar{\varepsilon}$. Then we have
$$\mathrm{Hom}_{G_{F}}(\rho, \hat{H}^1_{GG,E}(K^p)) \cong J'(B(\rho|_{G_{F_{v}}})) \otimes_{i} B(\rho|_{G_{F_{v_i}}})  \otimes AF_{\ell \neq p}$$
where $AF_{\ell \neq p} = \left(\bigotimes'_{u\nmid p\infty} \pi_{LL}(\rho|_{G_{F_u}})\right)^{K^p}$.
\end{theo}
\begin{proof}
Let $\pi(K^p) = \widehat{H}^0_{\overline{GG}, E}(\overline{K}^p)$ be the space of p-adic automorphic forms on $\overline{GG}$ of level $\overline{K}^p$. By Theorem 6.2 in \cite{sch} we have an isomorphism \footnote{Scholze actually considered torsion modules, but the proof is valid in the case of completed cohomology, as he remarks himself.}
$$\hat{H}^1_{GG,E}(K^p) \simeq H^1 _{et} (\mb{P}^1 _{C}, \mc{F}_{\pi(K^p)})$$
Let $\lambda: \mb{T}(K^p) \ra E$ be a Hecke system \footnote{One might also reason with $\mf{m}$, the maximal ideal associated to $\bar{\rho}$, and localise cohomology groups at $\mf{m}$.} associated to $\rho$. We now take $\mathrm{Hom}_{G_{F}}(\rho,-)$, and observe that $\mb{T}(K^p)$ acts through $\lambda$ on $\mathrm{Hom}_{G_{F}}(\rho,\hat{H}^1_E(K^p))$ (this is the Eichler-Shimura relation) and 
$$H^1 _{et} (\mb{P}^1 _{C}, \mc{F}_{\pi(K^p)})^{\mb{T}(K^p)=\lambda} \simeq H^1 _{et} (\mb{P}^1 _{C}, \mc{F}_{\pi(K^p)^{\mb{T}(K^p)=\lambda}})$$ 
as $\mb{T}(K^p)$ acts on $H^1 _{et} (\mb{P}^1 _{C}, \mc{F}_{\pi(K^p)})$ through $\pi(K^p)$. Hence we are left with proving an isomorphism of $\prod _i G_{F_{v_i}} \times G(\mb{A})$-modules
$$\Hom_{\prod _i G_{F _{v_i}}}(\otimes _i \rho _{| G_{F_{v_i}}}, H^1 _{et} (\mb{P}^1 _{C}, \mc{F}_{\pi(K^p)^{\mb{T}(K^p)=\lambda}}) \simeq $$
$$\simeq \Hom _{\prod _i G _{F _{v_i}}}(\otimes _i \rho _{| G_{F_{v_i}}} , H^1(\mb{P}^1 _C, \mc{F}_{\otimes _i B(\rho_{| G_{F_{v_i}}})})) \otimes \Hom(\otimes _i B(\rho_{| G_{F_{v_i}}}), \pi(K^p))$$
But this it true investigating $\mc{F}_{\pi}$ at the geometrical level and observing that we have
$$\pi(K^p)^{\mb{T}(K^p)=\lambda} = \bigotimes _i B(\rho_{| G_{F_{v_i}}}) \otimes \Hom(\otimes _i B(\rho_{| G_{F_{v_i}}}), \pi(K^p))$$  	
by Proposition \ref{prop-lcbar}.

\end{proof}

\noindent We observe that this is analogous to Theorem \ref{theo-lg1}, but for simpler representations. With these local-global compatibilities, we can now prove the main theorem using patching.

\section{Patching}

In this section we prove that both constructions give the same functor. Let us choose a representation $\overline{\rho}_p : G_{\bb{Q}_p} \rightarrow \mathrm{GL}_2(k_E)$.

\begin{theo}\label{theo:comp}
Let $(A,\mf m)$ be a complete noetherian local ring with finite residue field of characteristic $p$ and $V$ an admissible $A[\GL_2(\bb{Q}_p)]$-module such that $V / \mf m V \simeq B(\bar{\rho}_p)$. We have $J(V) = J'(V)$.
\end{theo}

\noindent Let $\overline{\rho}: G_{F'} \rightarrow \mathrm{GL}_2(k_E)$ be an essentially conjugate self-dual globalization of $\overline{\rho}_p$, that is, $\overline{\rho}|_{G_{{F'}_{v'}}} \cong \overline{\rho}_p$ (for all $v' | p$) and $\rho^{\vee} \cong \rho^c \otimes \chi$ for some character $\chi$.  Also, let $S$ be a set of primes of $F'$ containing all primes over $p$ and stable under conjugation.  Then there are deformation rings $R_{\overline{\rho}, S}^{\psi}$ and $R_{\overline{\rho}, S}^{\square, \psi}$ parameterizing (framed) deformations of $\rho$ that are also essentially conjugate self dual with determinant $\psi \chi_{cyc}$.  Additionally, there are local deformation rings $R_{\overline{\rho}_p}^{\psi_v}$, $R_{\overline{\rho}_p}^{\square, \psi_v}$, $R_{\overline{\rho}_{v_i}}^{\psi_{v_i}}$ and $R_{\overline{\rho}_{v_i}}^{\square, \psi_{v_i}}$ parameterizing (framed) local deformations.

\medskip

\noindent Put $g = \dim _{\mb{F}_q} H^1(G_{\mb{Q},S}, ad^0\bar{\rho}(1)) - [F:\mb{Q}]$. For the patching argument we fix finite sets $Q_N$ of $g + [F:\mb{Q}]$ primes $l$ of $F$ such that $q_l \equiv 1 \mod p^N$ for all $l \in Q_n$, $l$ splits completely in $F'$, and $\overline{Frob}_l$ has distinct eigenvalues (they exist by Proposition 2.2.4 in \cite{ki}). Moreover $R_{\bar{\rho}, S\cup Q_N} ^{\square, \psi}$ is topologically generated by $g$ elements over $R_{\overline{\rho}_p} ^{\square, \psi_v}\hat{\otimes} R_{\overline{\rho}_{v_1}}^{\square, \psi_{v_1}} \hat{\otimes} \cdots \hat{\otimes} R_{\overline{\rho}_{v_n}}^{\square, \psi_{v_n}}$. 

\medskip

\noindent For each $N \geq 1$, let $U_{Q_N}(1) \subset U_{Q_N}(0) \subset G(\mb{A} _f ^S) \simeq \GL_2(\mb{A} _f ^S)$ be the compact open subgroups given by
$$U_{Q_N}(1) = \prod _{l \not \in Q_N} \GL_2(\mc{O}_{F_l}) \times \prod _{l\in Q_N} U_l(1) \subset U_{Q_N}(0) = \prod _{l \not \in Q_N} \GL_2(\mc{O}_{F_l}) \times \prod _{l\in Q_N} U_l(0)$$
where 
$$U_l(1) = \{\left( \begin{smallmatrix} a & b \\ c & d \end{smallmatrix} \right) \ | \ c \equiv 0 \mod l, a/d \mapsto 1 \in \Delta _l \} $$
$$\subset U_l(0) = \{\left( \begin{smallmatrix} a & b \\ c & d \end{smallmatrix} \right) \ | \ c \equiv 0 \mod l \} \subset \GL_2(\mc{O}_{F_l})$$
where $\Delta _l \simeq \mb{Z} / p^N \mb{Z}$ is the unique quotient of order $p^N$ of the units $k_l ^{\times}$ of the residue field $k_l$ at $l$. Hence $U_{Q_N}(1) \subset U_{Q_N}(0)$ is a normal subgroup with quotient $\Delta _{Q_N} =  U_{Q_N}(0) / U_{Q_N}(1) \simeq (\mb{Z} / p^N \mb{Z}) ^{g +[F:\mb{Q}]}$.

\medskip

\noindent Up to replacing $\mb{F}_q$ by $\mb{F}_{q^2}$, we can fix a root $\alpha _v$ of the polynomial $X^2 - T_vX + q_v S_v$ in $\mb{F}_q$ for all $v\in Q_n$. For sufficiently small compact open subgroups $K \subset \GL_2(F)$ we set for $i=1,2$
$$S_{\psi}(KU_{Q_n}(i), \mc{O} _E) = C^0(G(\mb{Q}) \backslash G(\mb{A}_f) / KU_{Q_n}(i), \mc{O}_E)[\psi]$$
be the space of "automorphic" functions with central character $\psi$. On these spaces we can define an action of the Hecke algebra $\mb{T}(U_{Q_n}(i))$ generated by the usual elements $T_v$ and $S_v$ for $v \not \in Q_n$, $v \not \in S$, as well as operators $U_v$ for $v \in Q_n$ given by the double coset $[U_v(i)diag(\varpi _v, 1) U_v(i)]$. Let $\mf{m} _{Q_n}(i) \subset \mb{T}(U_{Q_n}(i))$ denote the maximal ideal generated by $\mf m \cap \mb{T}(U_{Q_n}(i))$ and $U_v - \alpha _v$ for $v \in Q_n$. By Lemma 2.1.7 of \cite{ki} the natural map
$$S_{\psi}(E, \mc{O}_E)_{\mf m} \ra S_{\psi}(KU_{Q_n}(0), \mc{O} _E) _{\mf{m} _{Q_n}(0)}$$
is an isomorphism. Lemma 2.1.4 of \cite{ki} implies that $S_{\psi}(KU_{Q_n}(1), \mc{O} _E) _{\mf{m} _{Q_n}(1)}$ is a finite free $\mc{O}_E[\Delta _{Q_n}]$-module with 
$$S_{\psi}(KU_{Q_n}(1), \mc{O} _E) _{\mf{m} _{Q_n}(1)} \otimes _{\mc{O}_E[\Delta _{Q_n}]} \mc{O} _E \simeq S_{\psi}(KU_{Q_n}(0), \mc{O} _E) _{\mf{m} _{Q_n}(0)}$$
By the existence of Galois representations, there is an action of the deformation ring $R ^{\psi} _{\bar{\rho}, S \cup Q_n}$ on $S_{\psi}(KU_{Q_n}(1), \mc{O} _E) _{\mf{m} _{Q_n}(1)}$. Moreover, using local deformation rings at places $v \in Q_n$, there is a map $\mc{O}_E[[y_1,...,y_{g + [F:\mb{Q}]}]] \ra R ^{\psi} _{\bar{\rho}, Q_n}$ such that the action of $\mc{O}_{E}[[y_1,...,y_{g+[F:\mb{Q}]}]]$ on $S_{\psi}(KU_{Q_n}(1), \mc{O} _E) _{\mf{m} _{Q_n}(1)}$ comes from the $\Delta _{Q_n}$-action via the fixed surjection 
$$\mc{O}_E[[y_1,...,y_{g+[F:\mb{Q}]}]] \ra \mc{O}_E[(\mb{Z}/p^n\mb{Z})^{g+[F:\mb{Q}]}] \simeq \mc{O}_E[\Delta _{Q_n}]$$
The map $R_{\bar{\rho}, S\cup Q_n} ^{\psi} \ra R_{\bar{\rho}, S\cup Q_n} ^{\square, \psi}$ is formally smooth of dimension 3, so that we can choose 
$$y_{g+[F:\mb{Q}] +1}, y_{g+[F:\mb{Q}] +2}, y_{g+[F:\mb{Q}] +3}$$ 
such that
$$R_{\bar{\rho}, S\cup Q_n} ^{\square, \psi} \simeq R_{\bar{\rho}, S\cup Q_n} ^{ \psi}[[ y_{g+[F:\mb{Q}] +1}, y_{g+[F:\mb{Q}] +2}, y_{g+[F:\mb{Q}]+3}   ]]$$ 
Let us also fix a surjection
$$R_{p} ^{\square, \psi}[[x_1,...,x_g]] \ra R_{\bar{\rho}, S\cup Q_n} ^{\square, \psi} $$
and a lifting
$$\mc{O}[[y_i]] \ra R_{p} ^{\square, \psi}[[x_1,...,x_g]]$$
where we write $\mc{O}[[y_i]]$ for $\mc{O}[[y_1,..., y_{g+[F:\mb{Q}]+3}]]$. We put
$$S_n(K) = R_{\bar{\rho}, S\cup Q_n} ^{\square, \psi} \otimes _{R_{\bar{\rho}, S\cup Q_n} ^{ \psi}} S_{\psi}(KU_{Q_n}(1), \mc{O}) _{\mf m _{Q_n}(1)}$$
which is a $R_{p} ^{\square, \psi}[[x_1,...,x_g]]$-module via the surjection $R_{p} ^{\square, \psi}[[x_1,...,x_g]] \ra R_{\bar{\rho}, S\cup Q_n} ^{\square, \psi}$.

\medskip

\noindent For an open ideal $I \subset \mc{O}_E[[y_i]]$ we define
$$\pi _n(I) = \varinjlim _K S_n(K) \otimes _{\mc{O}_E} \mc{O} _E[[y_i]] / I$$
We remark that for sufficiently large $n$ so that $I$ contains the kernel of 
$$\mc{O}_E[[y_i]] \ra \mc{O}_E[\Delta _{Q_n}][[y_{g+[F:\mb{Q}] +1}, y_{g+[F:\mb{Q}] +2}, y_{g+[F:\mb{Q}] +3}]]$$
$\pi_n(I)$ is an admissible $\prod _{v|p} \GL_2(F_v)$-representation over the finite ring $\mc{O}_{E}[[y_i]]/I$ such that $\pi _n(I) ^K$ is finite free for all sufficiently small compact open subgroups $K \subset \GL_2(F)$. Furthermore
$$\pi _n(I) \otimes _{\mc{O} _E[[y_i]] /I } \mc{O} _E / \varpi _E = C^0 (G(\mb{Q}) \backslash G(\mb{A}_f ^S) / \prod _{v\not \in S} \GL _2(\mc{O}_{F_v}), \mc{O}_E/\varpi _E)$$
is independent of $n$, in particular $\pi _n(I) ^K$ are bounded uniformly in $n$ and we can take the ultraproduct as in \cite{sch}. We fix a non-principal ultrafilter which gives us the localisation map
$$\prod _{n \geq 1 } \mc{O}_E[[y_i]] / I \ra \mc{O}_E[[y_i]]/I$$
We define
$$\pi _{\infty}(I) = \varinjlim _K \left( \prod _{n \geq 1} \pi _n(I) ^K \right) \otimes _{\prod _{n \geq 1} \mc{O} _{E}[[y_i]] / I} \mc{O}_E [[y_i]]/I$$
This is an admissible $\prod _{v|p} \GL_2(F_v)$-representation over $\mc{O}_E[[y_i]]/I$ such that
$$\pi _{\infty}(I)^K = \left( \prod _{n \geq 1} \pi _n(I) ^K \right) \otimes _{\prod _{n \geq 1} \mc{O} _{E}[[y_i]] / I} \mc{O}_E [[y_i]]/I$$
is finite free. We pass to the inverse limit
$$\pi _{\infty} '  = \varprojlim _I \pi _{\infty}(I)$$
and then set
$$\pi _{\infty} = \pi _{\infty} ' \otimes _{\mc{O}_E[[y_i]]} \omega$$
where $\omega$ is the injective hull of $\mc{O}_E / \varpi _E$ as $\mc{O}_E[[y_i]]$-module. This is an admissible $\prod _{v|p} \GL_2(F_v)$-representation over $\mc{O}_{E}[[y_i]]$. Observe that we have an action of $R^{\square, \psi} _{p}[[x_1,...x_g]]$ on all the above objects.

\medskip

\begin{proof}[Proof of Theorem \ref{theo:comp}]
We use patching argument to reduce to the case of mod p representations. Observe that it suffices to show the result for $A = R_{\overline{\rho}_p}[[t_1, \ldots, t_n]]$ for some $n$, $\rho_A$ corresponding to the map $A \rightarrow R_{\overline{\rho}_p}$ given by sending all the $t_i$s to $0$, and $V = B(\rho_A)$.

We globalize our representation $\bar{\rho}_p$. By Corollary A.3 of \cite{gk}, there is a CM field $F$ (potentially larger than $F$ we have started with, but still such that each place $v|p$ of $F^+$ splits in $F$) and an absolutely irreducible representation $\rho: G_{F} \rightarrow \mathrm{GL}_2(E)$ (again we might need to enlarge $E$) which is automorphic and such that $\overline{\rho}|_{G_{F_v}} = \overline{\rho}_p$ for each place $v|p$. We apply the patching construction to $\overline{\rho}$ to get a representation $\pi _{\infty}$ as above. In order to establish the theorem we need to prove that
$$J(\pi _{\infty,v}) = J'(\pi _{\infty, v})$$
where $v$ is our distinguished place over $p$ and we write $\pi_{\infty,v}$ for the $v$-component of $\pi_{\infty}$. As both functors commute with limits, it is enough to prove that 
$$J(\pi_{\infty}(I)_v) = J'(\pi _{\infty}(I)_v)$$
for an open ideal $I \subset \mc{O}[[y_i]]$. By Theorem \ref{theo-lg2} (cf. Theorem 6.2 of \cite{sch}), we know that 
$$\mathrm{Hom}_{G_{F}}(\rho, \hat{H}^1_{GG,E}(K^p)) \cong J'(\pi_n(I)_v) \otimes_{i} \pi_n(I)_{v_i}  \otimes AF_{\ell \neq p}$$
but by Theorem \ref{theo-lg1} (cf. Corollary 5.3.4 in \cite{EKThesis}), we have:
$$\mathrm{Hom}_{G_{F}}(\rho, \hat{H}^1_{GG,E}(K^p)) \cong J(\pi_n(I)_v) \otimes_{i} \pi_n(I)_{v_i}  \otimes AF_{\ell \neq p}$$
Thus both functors $J(-)$ and $J'(-)$ are equal on $\pi _n (I)_v$, hence by taking the ultra-product and passing to the limit, we obtain (cf. proof of Corollary 9.3 of \cite{sch}) that
$$J(\pi _{\infty,v}) = J'(\pi _{\infty,v})$$ 
This finishes the proof.
\end{proof}

Another way to think about this result is that there are three objects: the patched $H^1$ of the Shimura curve, and $J$,$J'$ functors applied to the patched $H^0$ for the group that is compact at $\infty$. Then this theorem shows that either of the last two objects agrees with the first object, and so the last two objects must agree with each other. One benefit of thinking in this manner is that if there is another $J$ functor that satisifies a similar local-global compatibility, then it would also agree with the patched $H^1$, and so with the two already constructed functors.

\section{Locally algebraic vectors}

In this section, we compute the locally algebraic vectors in $J(\rho)$.  The answer is highly reminiscent of the Breuil-Schneider construction for the case of $\mathrm{GL}_n$.  If $\rho : G_{\bb{Q}_p} \rightarrow \mathrm{GL}_2(E)$ is a continuous representation of $G_{\bb{Q}_p}$, then define $BS_{D^\times}(\rho)$ as follows: if $\rho$ is not potentially semistable with distinct Hodge-Tate weights, then $BS_{D^\times}(\rho) = 0$.  Otherwise, associated to $\rho$ are Hodge-Tate weights $w_1 < w_2$ and a Weil-Deligne representation $WD(\rho)$.  If the Frobenius-semisimplification of $WD(\rho)$ is the sum of two characters, then, again, define $BS_{D^\times}(\rho) = 0$.  In the final case, we let $Sm_{\rho}$ be the representation of $D^\times$ associated to $WD(\rho)^{F-ss}$ and $Alg_{\rho}$ be the algebraic representation of $D^\times$ with weights $-w_2$ and $-w_1 - 1$.  Then we may define $BS_{D^\times}(\rho) = Sm_{\rho} \otimes Alg_{\rho}$.

\begin{theo}
We have $J(\rho)^{alg} = BS_{D^\times}(\rho)$.
\end{theo}
The proof will break up into three steps.  First off, we define the potentially semistable deformation rings that we are interested in.  After that we use patching to reduce showing the theorem for all representations to showing the theorem for globally arising representations.  Finally, we show that the theorem is true for globally arising representations.
\begin{proof}
The first part of the proof will break into two cases.  The first case will be when $\rho$ is potentially crystalline, and the second will be when $\rho$ is semistable.  Since every potentially semistable $2$-dimensional representation is either potentially crystalline or a twist of a semistable representation by a finite order character, there is no loss in generality in these assumptions.

Let $\overline{\rho}$ be a representation of $G_{\bb{Q}_p}$ over $k$.  Then, we fix an inertial type $\tau$ and a pair of Hodge-Tate weights $w_1 < w_2$.  There is then a deformation ring $R_{\overline{\rho}}^{\tau, w_1, w_2}/\mathcal{O}_E$ whose points correspond to potentially crystalline lifts $\rho$ of $\overline{\rho}$ such that $\rho$ has Hodge-Tate weights $w_1<w_2$ and the inertial type of $WD(\rho)$ is $\tau$.  Finally, we will assume that $\tau$ is chosen such that $WD(\rho)$ is irreducible for any lift $\rho$.

If one instead chooses $\tau$ to be trivial, then one gets a ring $R_{\overline{\rho}}^{ss, w_1, w_2}$ which parameterizes semistable lifts of $\overline{\rho}$.  If $S = \Spec(R_{\overline{\rho}}^{ss, w_1, w_2})$, there is the locus $S^{N \neq 0}$ which is open, and corresponds to the condition that $\rho$ is not crystalline.  Additionally, there is the closed locus $S'$ where the Weil group representation is of the form $\left(\begin{smallmatrix} \chi & \\ & \chi\epsilon \end{smallmatrix}\right)$.  One knows that $S'$ is the closure of $S^{N \neq 0}$, and we will let $R_{\overline{\rho}}^{ss', w_1, w_2}$ be the quotient of $R_{\overline{\rho}}^{ss, w_1, w_2}$ corresponding to $S'$.  Again, one has that there is a representation $BS_{D^\times}(\rho^{univ})$ over $R_{\overline{\rho}}^{ss', w_1, w_2}$.  An important thing to note here is that there could be representations $\rho$ which are crystalline but still lie in $S'$.  These representations arise from the fact that if $\rho$ is crystalline, and $WD(\rho)=\left(\begin{smallmatrix} \chi & \\ & \chi\epsilon\end{smallmatrix}\right)$, then even though $N = 0$ at $\rho$, it deforms to representations where $N \neq 0$.  However, such representations should not arise from global geometric (or modular) representations, as they are in contradiction with the weight-monodromy conjecture.

\begin{lemm}
Every irreducible component of $\Spec(R_{\overline{\rho}}^{\tau, w_1, w_2})$ and $\Spec(R_{\overline{\rho}}^{ss', w_1, w_2})$ contains an automorphic point.
\end{lemm}
\begin{proof}
Let $\sigma$ be the locally algebraic type corresponding to the potentially semistable representation we are looking at, i.e. $\sigma$ is a locally algebraic representation of $\mathrm{GL}_2(\bb{Z}_p)$.  Adopting the notation of \cite{ceggps1}, there is a module $M_{\infty}(\sigma^\circ)$ over the ring $R_{\infty}$.  One has in addition that there is an ideal $\mathfrak{a}$ that is generated by a regular sequence, and such that any point in the support of $M_{\infty}(\sigma^\circ)/\mathfrak{a}$ corresponds to an automorphic form of the correct inertial type and weight.  Thus, one has that every component in the support of $M_{\infty}(\sigma^\circ)$ contains an automorphic point (see Lemma 3.9 of \cite{Pas} for the short commutative algebra argument).  However, the support of $M_{\infty}(\sigma^{\circ})$ is exactly the set of points that contain the correct inertial type by definition.  In \cite{ceggps2}, they show that $M_{\infty}$ realizes the standard $p$-adic Langlands correspondence, and so one gets that every point corresponding to a potentially semistable representation is in the support of $M_\infty(\sigma^\circ)$.  This shows the lemma.
\end{proof}
The above argument is very similar in spirit to how Emerton shows the Fontaine-Mazur conjecture for $2$-dimensional representations of $G_{\bb{Q}}$ in \cite{em1}.

Now, one considers $\mu$, the smooth representation of $\mathcal{O}_D^\times$ associated to $\tau$ defined by taking some Weil-Deligne representation that has $\tau$ as its inertial type, taking the representation of $D^\times$ associated to said representation and then restricting to $\mathcal{O}_D^\times$.  It is a standard fact that this restriction depends only on $\tau$ and so $\mu$ is well-defined.  As in \cite{ceggps1}, one considers $\Hom_{\mathcal{O}_D^\times}(\mu\otimes Alg_{-w_2, -w_1-1}, J(\pi _\infty)).$  A similar argument as above shows that the support of this is the union of connected components in $\Spec(R^{\tau, w_1, w_2}_{\overline{\rho}}[[y_i]]$ (or $\Spec(R^{ss', w_1, w_2}_{\overline{\rho}}[[y_i]]$).  Granting for the moment that one has the correct locally algebraic vectors at all of the automorphic points, one gets that $\Hom_{\mathcal{O}_D^\times}(\mu\otimes Alg_{-w_2, -w_1-1}, J(\pi _\infty))$ is nonzero over all of $R^{\tau, w_1, w_2}_{\overline{\rho}}[[y_i]]$ or $R^{ss', w_1, w_2}_{\overline{\rho}}[[y_i]]$.  This shows that one has the correct representation (correct locally algebraic vectors) when restricted to $\mathcal{O}_D^\times$.  Additionally, since one can read the central character of $J(\pi_\infty)$ off of $\det(\rho_\infty)$, one gets that, over $R^{\tau, w_1, w_2}_{\overline{\rho}}[[y_i]]$ or $R^{ss', w_1, w_2}_{\overline{\rho}}[[y_i]]$, the locally algebraic vectors are correct when restricted to $\mathcal{O}_D^\times\bb{Q}_p^\times$. We want to promote that to $D^{\times}$. Consider the following lemma:

\begin{lemm}\label{SmoothRepLemma} Let $A$ be an $\overline{E}$-algebra that is a domain.
\begin{enumerate}
\item Let $\pi: D^\times \rightarrow \mathrm{GL}_n(A)$ be a smooth irreducible representation that is constant on $\mathcal{O}_D^\times\bb{Q}_p^\times$.  Then $\pi$ is constant.
\item Similarly, let $WD$ be a $2$-dimensional irreducible Weil-Deligne representation over $A$.  If the inertial type and determinant are constant, then $WD$ is constant.
\end{enumerate}
\end{lemm}
Here, $\pi$ being constant means that there is a representation $\pi'$ over $\overline{E}$ such that $\pi = \pi' \otimes_{\overline{E}} A$.
\begin{proof}
The first part breaks into two cases.  Let $\pi'$ be a representation of $\mathcal{O}_D^\times \bb{Q}_p^\times$ over $\overline{E}$ such that $\pi = \pi' \otimes_{\overline{E}} A$.  The two cases are whether $\pi'$ is irreducible or not.

If $\pi' = \pi_1 \oplus \pi_2$, then Frobenius reciprocity tells you that $\mathrm{Ind}_{\mathcal{O}_D^\times \bb{Q}_p^\times}^{D^\times} \pi_1 \cong \mathrm{Ind}_{\mathcal{O}_D^\times \bb{Q}_p^\times}^{D^\times} \pi_2$ and that $\mathrm{Hom}_{D^\times}(\mathrm{Ind}_{\mathcal{O}_D^\times \bb{Q}_p^\times}^{D^\times} \pi_1 \otimes_{\overline{E}} A, \pi)$ is nontrivial and hence contains an isomorphism.  Then one is done, as the representation $\mathrm{Ind}_{\mathcal{O}_D^\times \bb{Q}_p^\times}^{D^\times} \pi_1 \otimes_{\overline{E}} A$ is visibly constant.

Now, assume that $\pi'$ is irreducible.  Since $\pi'$ admits an extension to an irreducible representation of $D^\times$, one has that $\mathrm{Ind}_{\mathcal{O}_D^\times \bb{Q}_p^\times}^{D^\times} \pi'$ is the sum of two irreducible representations; call them $\pi'_1$ and $\pi'_2$.  Then one gets that either $\mathrm{Hom}(\pi'_1\otimes_{\overline{E}} A, \pi)$ or $\mathrm{Hom}(\pi'_2\otimes_{\overline{E}} A, \pi)$ is nonzero and hence an isomorphism.  Again, one has that $\pi_i \otimes_{\overline{E}} A$ is a constant representation and hence so is $\pi$.

The second part is simpler.  Because of the irreducibility of $WD$, one knows that this is the induction of a character $\psi$ of some extension $L/\bb{Q}_p$.  It then suffices to show that this character is constant.  One knows that $\psi$ is constant on $\mathcal{O}_L^\times$ because the inertial type is constant there.  It then reduces to showing that $\psi(\varpi_L) \in \overline{E}^\times$.  Because $A$ is a domain, it is in fact necessary only to show that $\psi(\varpi_L) \in \overline{E}^\times$.

If $L$ is unramified, then one may choose the uniformizer of $\mathcal{O}_L$ to be $p$.  In this case, one has that $\psi(p)^2 = \mathrm{det}(\mathrm{Ind}_L^{\bb{Q}_p}\psi)(p^2) = \mathrm{det}(WD)(p^2)$.  If $L/\bb{Q}_p$ is ramified, then one may choose $\varpi_L$ such that $\varpi_L^2 \in \bb{Q}_p$.  Then one has that $\psi(\varpi_L)^2 = \psi(-1)\psi(\varpi_L)\psi(-\varpi_L) = \psi(-1)\mathrm{det}(\mathrm{Ind}_L^{\bb{Q}_p}(\psi))(N_{\bb{Q}_p}^L(\varpi_L)) = \psi(-1)\mathrm{det}(WD)(N_{\bb{Q}_p}^L(\varpi_L))$, which again is constant.
\end{proof}

Let $\{X_i\} _i$ be the geometric components of $\Spec(R^{\tau, w_1, w_2}_{\overline{\rho}}[[y_i]])$.  Over each $X_i$, there is a Weil-Deligne representation $WD_{X_i}$ given by a recipie of Fontaine.  Because the inertial type is constant, one has that $WD_{X_i}$ is of the form $\chi \otimes \psi$ where $\chi$ is an unramified character and $\psi$ has constant determinant.  But then by the second part of lemma \ref{SmoothRepLemma}, $\psi$ is constant, and so there is a smooth representation of $D^\times$ associated to $\psi$.  Then, one can tensor it by $\chi \circ \det$ to get a smooth representation $Sm_{X_i}$ such that for each closed point $p$ of $X_i$, $Sm_{X_i}(p)$ is the $D^\times$-representation associated to $WD_{X_i}$ under the local Langlands and Jacquet-Langlands correspondence.

Now, we claim that $Sm_{X_i} \otimes Alg_{-w_2, -w_1-1}$ is equal to the representation $J(\pi_\infty|_{X_i})^{alg}$.  Indeed, one has that this is true when restricted to $\mathcal{O}_D^\times \bb{Q}_p^\times$.  Again, we will use the fact that this is true at all automorphic points as well, to show that this equality is true at at least one closed point.  Thus, by the first part of lemma \ref{SmoothRepLemma} and the same manipulations as in the previous paragraph, one gets that these two representations are the tensor product of a fixed algebraic representation, a fixed unramified character, and a representation that is constant along $\mathcal{O}_D^\times \bb{Q}_p^\times$.  Hence, by the first part of lemma \ref{SmoothRepLemma}, both representations are the tensor product of a constant representation and a fixed unramified character.  But since these representations agree at a closed point (namely, an automorphic point which we know to exist), they must be the same everywhere.

We thus are reduced to showing that, if $\rho: G_{F'} \rightarrow \mathrm{GL}_2(E)$ is a promodular representation that satisfies $\rho^c = \chi \otimes \rho^\vee$, then $BS_{D^\times}(\rho|_{G_{F_{v'}}}) = J(B(\rho|_{G_{F_{v'}}}))^{alg}$.  This is in turn reduced to calculating $\hat{H}^1_{E, GG}(K^p)^{alg}$ for $K^p \subset GG(\mathbb{A}_{f}^{p})$.  Now, Emerton's paper \cite{em2} gives a technique for doing so.  Let $W$ be an algebraic representation of $GG(\bb{Q}_p)$.  Associated to $W$ is a local system $\mathcal{V}_{W^\vee} / Sh_{K_pK^p}$.  Define $H^i(W^\vee, K^p) = \displaystyle{\lim_{\longrightarrow \atop K_p}} H^i_{\acute{e}t}(Sh_{K_pK^p}, \mathcal{V}_{W^\vee})$.  Then, \cite{em2} gives a spectral sequence $\Ext^i_{\mathfrak{gg}}(W, \hat{H}^j_{E, GG}(K^p)^{an}) \Rightarrow H^{i+j}(W^\vee, K^p)$.

Analogous to the $\mathrm{GL}_2$ case, there is a decomposition of $\mathfrak{gg} = \mathfrak{z} \oplus \mathfrak{sd} \oplus \bigoplus_{i = 1}^{n} \mathfrak{sl}_{2, v_i} := \mathfrak{z} \oplus \mathfrak{gg}^{ss}$, where $\mathfrak{z}$ is the center of the Lie algebra, $\mathfrak{sd}$ is the Lie algebra of $SD^\times = \{d \in D^\times | \nu(d) = 1\}$, and $\mathfrak{sl}_{2, v_i}$ is the copy of $\mathfrak{sl}_2$ corresponding to the $SL_2(F_v)$ inside of $GG(\bb{Q}_p)$.  After possibly replacing $E$ with a larger field, one has that $\mathfrak{sd}_E \cong \mathfrak{sl}_{2, E}$, and since all of our representations are $E$ linear, then one gets that $\mathfrak{gg} \cong \mathfrak{z} \oplus \mathfrak{sl}_2^{n+1}$ after base changing to $E$.

Because the maximal $\bb{R}$-split central torus of $GG$ is $\bb{Q}$-split, one has that, for $Z \subset Z(GG(\bb{Q}_p))$ sufficiently small, $\hat{H}^0_{E,GG}(K^p) \cong \mathcal{C}^0(G, E)^r$ as representations of $Z$ for some natural number $r$.  Moreover, the action of the semisimple part of $GG(\bb{Q}_p)$ on $\hat{H}^0_{E, GG}(K^p)$ is trivial, so one gets that $\hat{H}^0_{E, GG}(K^p)^{an} \cong \mathcal{C}^0(Z, E)^{an} \boxtimes 1$ as representations of $\mathfrak{gg} = \mathfrak{z} \oplus \mathfrak{gg}^{ss}$.  Since $\mathfrak{gg}^{ss}$ is geometrically the sum of copies of $\mathfrak{sl}_2$, a Kunneth formula says that $H^i(\mathfrak{gg}^{ss}, W^{\vee}) := \Ext^i_{\mathfrak{gg}^{ss}}(W, 1)$ is concentrated in degrees $0, 3, \ldots, 3(n+1)$ (as the cohomology of $\mathfrak{sl}_2$ is concentrated in degrees $0$ and $3$).  Additionally, a Kunneth formula also says that $\Ext^i_{\mathfrak{gg}}(W, \hat{H}^0_{E, GG}(K^p)^{an}) = \oplus_{a+b = i} \Ext^a_{\mathfrak{z}}(\chi, (\mathcal{C}^0(Z, E)^{an})^r) \oplus H^b(\mathfrak{gg}^{ss}, W^{\vee})$.  However, one gets that the only $\mathfrak{z}$ term that contributes is the degree $0$ part due to the freeness of the module, and using the aforementioned results about Lie algebra cohomology, one sees that the $\Ext^1(H^0)$ and the $\Ext^2(H^0)$ terms vanish, so one gets an isomorphism $H^1(W^\vee) \cong \Hom_{\mathfrak{gg}}(W, \hat{H}^1_{E, GG}(K^p))$.  But that in turn gives an isomorphism $\hat{H}^1_{E, GG}(K^p)^{W-alg} \cong W \otimes H^1(W^\vee, K^p)$. 

Now, we can use classical local-global compatibility results, which describe $H^1(W^\vee, K^p) = \bigoplus_{\rho} \chi_{\rho} \otimes JL(WD(\rho|_{G_{F_{v'}}})) \otimes \bigotimes \pi_{LL}(\rho|_{G_{F_{v_i}}})$, where the sum is over all $\rho$ that are modular of weight $W^{\vee}$ and tame level $K^p$.  Thus, one has that, for $\rho$ modular, $J(\rho|_{G_{F_{v'}}})^{alg} = BS_{D^\times}(\rho|_{G_{F_{v'}}})$, which is what we needed.  
\end{proof}


\begin{thebibliography}{9}

\bibitem[BGG]{bgg} T. Barnet-Lamb, T. Gee, D. Geraghty, "Serre weights for rank two unitary groups", Mathematische Annalen 356 (2013), no. 4, 1551-1598.

\bibitem[BC]{bc} J. Bergdall, P. Chojecki, "Ordinary representations and companion points for U(3) in the indecomposable case", preprint 2014

\bibitem[BH]{bh} C. Breuil, F. Herzig, "Ordinary representations of G(Qp) and fundamental algebraic representations", Duke Math. J. 164, 2015, 1271-1352.

\bibitem[BP]{bp} C. Breuil, V. Paskunas, "Towards a modulo p Langlands correspondence for $\GL_2$", Memoirs of Amer. Math. Soc. 216, 2012.

\bibitem[CEGGPS1]{ceggps1} A. Caraiani, M. Emerton, T. Gee, D. Geraghty, V. Paskunas, and S. W. Shin, "Patching and the $p$-adic local Langlands correspondence", Camb. J. Math. 4 (2016), no.2, 197-287. 

\bibitem[CEGGPS2]{ceggps2} A. Caraiani, M. Emerton, T. Gee, D. Geraghty, V. Paskunas, and S. W. Shin, "Patching and the $p$-adic Langlands program for $\GL(2, \mathbb{Q}_p)$", preprint (2016).

\bibitem[Cho1]{cho1} P. Chojecki, "`On mod $p$ non-abelian Lubin-Tate theory for $\GL_2(\mathbb{Q}_p)$"', Compositio Math., volume 151, issue 08 (2015), pp. 1433-1461.

\bibitem[Cho2]{cho2} P. Chojecki, "On non-abelian Lubin-Tate theory and analytic cohomology", to appear in Proceedings of the AMS.

\bibitem[Cho3]{cho3} P. Chojecki, "Weak local-global compatibility and ordinary representations", to appear in Rendiconti del Seminario Matematico della Universit\`{a} di Padova.

\bibitem[CS1]{cs1} P. Chojecki, C. Sorensen, "Weak local-global compatibility in the $p$-adic Langlands program for $U(2)$", Rend. Sem. Mat. Uni. Padova., vol 137 (2017), pp. 101-133.

\bibitem[CS2]{cs2} P. Chojecki, C. Sorensen, "Strong local-global compatibility in the $p$-adic Langlands program for $U(2)$", Rend. Sem. Mat. Uni. Padova., vol 137 (2017), pp. 135-153.

\bibitem[Em1]{em1} M. Emerton, "Local-global compatibility in the p-adic Langlands programme for $\GL_{2 / \mathbb{Q} }$", preprint (2011).

\bibitem[Em2]{em2} M. Emerton, "On the interpolation of systems of eigenvalues attached to automorphic Hecke eigenforms", Invent. Math. 164, no. 1, 1-84 (2006).

\bibitem[GK]{gk} T. Gee, M. Kisin, "The Breuil-Mezard conjecture for potentially Barsotti-Tate representations", Forum of Math, Pi 2 (2014), e1.

\bibitem[Ki]{ki} M. Kisin, "The Fontaine-Mazur conjecture for $\GL_2$", J. Amer. Math. Soc., 22(3), pp. 641-690 (2009).

\bibitem[Kn]{EKThesis} E. Knight, PhD thesis, Harvard 2016.

\bibitem[Pas]{Pas} V. Paskunas, "On the image of Colmez's Montreal functor", Pub. Math. IHES 118 (2013), 1-119.

\bibitem[Sch1]{sch} P. Scholze, "On the p-adic cohomology of the Lubin-Tate tower", to appear in Annales de l'ENS.

\bibitem[Sch2]{sch2} P. Scholze, "On torsion in the cohomology of locally symmetric varieties", Annals of Mathematics 182 (2015), no. 3, 945-1066.

\bibitem[Sch3]{sch3} P. Scholze, "Perfectoid spaces", Publ. math. de l'IHES 116 (2012), no. 1, pp. 245-313.

\end{thebibliography}
\end{document}